\documentclass{lms_modified}

\newtheorem{theorem}{Theorem}[section] 
\newtheorem{lemma}[theorem]{Lemma}     

\newnumbered{assertion}{Assertion}    
\newnumbered{conjecture}{Conjecture}  
\newnumbered{definition}{Definition}
\newnumbered{hypothesis}{Hypothesis}
\newnumbered{remark}{Remark}
\newnumbered{note}{Note}
\newnumbered{observation}{Observation}
\newnumbered{problem}{Problem}
\newnumbered{question}{Question}
\newnumbered{algorithm}{Algorithm}
\newnumbered{example}{Example}
\newunnumbered{notation}{Notation} 

\usepackage{graphicx,amssymb,amstext,amsmath}

\usepackage{algorithm}
\usepackage{algpseudocode}
\algnewcommand\algorithmicinput{\textbf{Input:}}
\algnewcommand\INPUT{\item[\algorithmicinput]}
\algnewcommand\algorithmicoutput{\textbf{Output:}}
\algnewcommand\OUTPUT{\item[\algorithmicoutput]}
\algnewcommand\algorithmicoracle{\textbf{Oracle:}}
\algnewcommand\ORACLE{\item[\algorithmicoracle]}
\makeatletter
\algnewcommand{\LineComment}[1]{\Statex \hskip\ALG@thistlm \(\triangleright\) #1}
\makeatother

\newcommand{\p}{\mathfrak{p}}
\newcommand{\Z}{\mathbb Z}

\newcommand{\F}{\mathbb F}

\newcommand{\f}{\mathfrak f}

\newcommand{\D}{\Delta}

\usepackage{tikz}
\usetikzlibrary{matrix,arrows}

\newcommand{\ph}{(\phi/\p)}

\title[Factoring Polynomials using Drinfeld Modules]{Factoring Polynomials over Finite Fields  using Drinfeld Modules with Complex Multiplication}

\author{Anand Kumar Narayanan}

\begin{document}
\maketitle

\begin{abstract}
We present novel algorithms to factor polynomials over a finite field $\F_q$ of odd characteristic using rank $2$ Drinfeld modules with complex multiplication. The main idea is to compute a lift of the Hasse invariant (modulo the polynomial $f(x) \in \F_q[x]$ to be factored) with respect to a Drinfeld module $\phi$ with complex multiplication. Factors of $f(x)$ supported on prime ideals with supersingular reduction at $\phi$ have vanishing Hasse invariant and can be separated from the rest. A Drinfeld module analogue of Deligne's congruence plays a key role in computing the Hasse invariant lift. \\ \\
We present two algorithms based on this idea. The first algorithm  chooses Drinfeld modules with complex multiplication at random and has a quadratic expected run time. The second is a deterministic algorithm with $O(\sqrt{p})$ run time dependence on the characteristic $p$ of $\F_q$.
\end{abstract}
\section{Introduction}
Let $q$ be a power of a an odd prime $p$ and let $\F_q$ denote the finite field with $q$ elements. The univariate polynomial factorization problem over $\F_q$ is,\\
\begin{itemize}
 \item \textsc{Polynomial Factorization:} \textit{Given a monic square free $f(x) \in \mathbb{F}_q[x]$ of degree $n$, write $f(x)$ as a product of its monic irreducible factors.}\\
\end{itemize}
A square free polynomial is one that does not contain a square of an irreducible polynomial as a factor. The square free input assumption is without loss of generality \cite{knu,yun}. Berlekamp showed that \textsc{Polynomial Factorization} can be solved in randomized polynomial time \cite{ber} and there is an extensive line of research \cite{cz,gs,ks} leading to the fastest known algorithm \cite{ku} with expected run time $\widetilde{O}(n^{3/2} \log q + n \log^2q)$. The soft $\widetilde{O}$ notation suppresses $n^{o(1)}$ and $\log^{o(1)} q$ terms for ease of exposition.\\ \\
The use of Drinfeld modules to factor polynomials over finite fields originated with Panchishkin and Potemine \cite{pp} whose algorithm was rediscovered by van der Heiden \cite{vdH}. These algorithms, along with the author's Drinfeld module black box Berlekamp algorithm \cite{nar} are in spirit Drinfeld module analogues of Lenstra's elliptic curve method to factor integers \cite{len}. The Drinfeld module degree estimation algorithm of \cite{nar} uses Euler-Poincare charactersitics of Drinfeld modules to estimate the factor degrees in distinct degree factorization. A feature common to the aforementioned algorithms is their use of random Drinfeld modules, which typically don't have complex multiplication.\\ \\
Our first algorithm for \textsc{Polynomial Factorization} is a randomized algorithm with $$\widetilde{O}(n^2 \log q + n \log^2 q)$$ expected run time. The novelty is the use of Drinfeld modules with complex multiplication.\\ \\
The algorithm constructs a random rank $2$ Drinfeld module $\phi$ with complex multiplication by an imaginary quadratic extension of the rational function field $\F_q(x)$  with class number $1$. At roughly half of the prime ideals $\p$ in $\F_q[x]$, $\phi$ has supersingular reduction. The Hasse invariant of $\phi$ at a prime ideal $\p$ vanishes if and only if $\phi$ has supersingular reduction at $\p$. A Drinfeld module analogue of Deligne's congruence, due to Gekeler \cite{gek}, allows us to compute a certain lift of Hasse invariants modulo the polynomial $f(x)$ we are attempting to factor. As a consequence, this lift vanishes exactly modulo the irreducible factors of $f(x)$ that correspond to the primes with supersingular reduction. In summary, we get to separate the irreducible factors corresponding to primes with supersingular reductions from those with ordinary reduction.\\ \\
The algorithm itself is stated in a simple iterative form with no reference to Drinfeld modules although Drinfeld modules are critical in its conception and analysis. The run time complexity is identical to that of the commonly used Cantor-Zassenhaus \cite{cz} algorithm. Although slower than subquadratic time algorithms such as \cite{ks,ku}, unlike these subquadratic algorithms we do not rely on fast modular composition or fast matrix multiplication. This should make our algorithm easy to implement in practice.\\ \\
The question of whether \textsc{Polynomial Factorization} is in deterministic polynomial time is a central outstanding open problem. Berlekamp \cite{ber1} reduced \textsc{Polynomial Factorization} to finding roots of a polynomial in a prime order finite field and through this reduction proposed a deterministic algorithm. Shoup was the first to prove rigorous unconditional run time bounds for Berlekamp's deterministic algorithm and its variants \cite{sho}. The most difficult and hence interesting setting for these deterministic algorithms is when the underlying field $\F_q$ is of large prime order. That is $q=p$ is a large prime. In this case, Shoup's bound establishes the best known unconditional deterministic run time bound of $\widetilde{O}(\sqrt{p})$ suppressing the dependence on the degree of the polynomial factored. Assuming the generalized Riemann hypothesis, quasi polynomial time algorithms (c.f.\cite{evd}) and in certain special cases, polynomial time algorithms (c.f. \cite{hua}) are known. \\ \\
Our second algorithm is an unconditional deterministic algorithm for \textsc{Polynomial Factorization} using Drinfeld modules with complex multiplication. Instead of picking Drinfeld modules with complex multiplication at random, we fix a natural ordering. The rest of the algorithm is nearly identical to the earlier randomized version. 
Remarkably, we were able to translate Shoup's proof to apply to our algorithm. We prove a worst case running time of $\widetilde{O}(\sqrt{p})$, again suppressing the dependence on the degree of the polynomial factored.\\ \\
Curiously, our deterministic algorithm applies directly to the equal degree factorization problem: that is to factor a given a monic square free polynomial all of whose irreducible factors are of the same degree. Prior deterministic algorithms relied on reductions to the root finding problem, either through linear algebra or other means.\\ \\
The paper is organized as follows. In \S~\ref{drinfeld_section}, Drinfeld modules are introduced and the general algorithmic strategy is outlined with emphasis on the role played by Hasse invariants and Deligne's congruence. In \S~\ref{randomized_section}, the efficient construction of Drinfeld modules with complex multiplication is presented followed by the description of our first algorithm. The randomized algorithm is then rigorously analyzed using function field arithmetic. In \S~\ref{deterministic_section}, we present and analyze the deterministic version.
\section{Rank-2 Drinfeld Modules}\label{drinfeld_section}
\noindent Let $A=\F_q[x]$ denote the polynomial ring in the indeterminate $x$ and let $K$ be a field with a non zero ring homomorphism $\gamma:A \rightarrow K$. Necessarily, $K$ contains $\F_q$ as a subfield. Fix an algebraic closure $\bar{K}$ of $K$ and let $\tau: \bar{K} \longrightarrow \bar{K}$ denote the $q^{th}$ power Frobenius endomorphism. The ring of endomorphisms of the additive group scheme $\mathbb{G}_a$ over $K$ can be identified with the skew polynomial ring $K\langle \tau \rangle$ where $\tau$ satisfies the commutation rule $\forall u \in K, \tau u = u^q \tau$.
A rank-2 Drinfeld module over $K$ is (the $A$-module structure on $\mathbb{G}_a$ given by) a ring homomorphism
$$\phi : A \longrightarrow K\langle \tau \rangle$$ 
$$\ \ \ \ \ \ \ \ \ \ \ \ \ \ \ \ \ \ \ \ \ \ x \longmapsto \gamma(x) + g_\phi \tau + \Delta_\phi \tau^2$$
for some $g_\phi \in K$ and $\Delta_\phi \in K^\times$. For $a \in A$, let $\phi_a$ denote the image of $a$ under $\phi$. We will concern ourselves primarily with rank $2$ Drinfeld modules and unless otherwise noted, a Drinfeld module will mean a rank $2$ Drinfeld module.\\ \\
Henceforth, we restrict our attention to Drinfeld modules $\phi:A \longrightarrow \F_q(x)\langle\tau\rangle$ over $\F_q(x)$ (with $\gamma : A \rightarrow \F_q(x)$ being the inclusion (identity), $g_\phi(x)\in A$ and $\Delta_\phi(x)\in A^\times$) and their reductions.\\ \\
For a prime ideal $\p \subset A$, if $\Delta_\phi$ is non zero modulo $\p$, then the reduction $\phi/\p := \phi \otimes A/\p$ of $\phi$ at $\p$ is defined through the ring homomorphism $$\phi/\p : A \longrightarrow \F_\p\langle \tau \rangle$$ $$\ \ \ \ \ \ \ \ \ \ \ \ \ \ \ \ \ \ \ \ \ \ \ \ \ \ \ \ \ \ \ \ \ \ \ \ \ \ \ \ \ \ \ \ \ \ \ \ t \longmapsto t + (g_\phi \mod \p) \tau + (\Delta_\phi\mod\p) \tau^2$$
and the image of $a \in A$ under $\phi/\p$ is denoted by $(\phi/\p)_a$. Even if $\Delta_{\phi}$ is zero modulo $\p$, one can still obtain the reduction $(\phi/\p)$ of $\phi$ at $\p$ through minimal models of $\phi$ (c.f. \cite{gek1}). We refrain from addressing this case since our algorithms do not require it.\\ \\
For $f(x) \in A$, denote by $(f(x))$ the ideal generated by $f(x)$ and by $\deg(f)$ the degree of $f(x)$. For a non zero ideal $\f \subset A$, let $\deg(\f)$ denote the degree of its monic generator. For $f(x),g(x) \in A$, by $\gcd(f(x),g(x))$ we mean the monic generator of the ideal generated by $f(x)$ and $g(x)$. Abusing notation, by $\gcd(\alpha,f(x))$ for some $f(x) \in A, \alpha \in A/(f(x))$ we really mean the gcd of $f(x)$ and a lift of $\alpha$ to $A$.
\section{Hasse Invariants and Deligne's Congruence}
Let $\p \subset A$ be a prime not dividing $\Delta_\phi$. Let $p \in A$ be the monic generator of $\p$. The Hasse invariant $h_{\phi,\p}(x) \in A$ of $\phi$ at $\p$ is the coefficient of $\tau^{\deg(p)}$ in the expansion $$\ph_p  = \sum_{i=0}^{2\deg(p)} h_i(\ph)(x) \tau^i \in A\langle \tau \rangle.$$
The Drinfeld module $\phi$ has supersingular reduction at $\p$ if and only if $\p$ divides $(h_{\phi,\p}(x))$ \cite{gos}. If the choice of $\phi$ is clear from context, we will call $\p$ supersingular.\\ \\
Recursively define a sequence $(r_{\phi,k}(x) \in A,k \in \mathbb{N})$ as $r_{\phi,0}(x):=1$, $r_{\phi,1}(x):=g_\phi$ and for $m>1$,
\begin{equation}\label{eisenstein_recurrence}
r_{\phi,m}(x) := \left(g_\phi(x)\right)^{q^{m-1}}r_{\phi,m-1}(x) - (x^{q^{m-1}}-x) \left(\D_\phi(x)\right)^{q^{m-2}} r_{\phi,m-2}(x) \in A.
\end{equation}
Gekeler (c.f.\cite{gek}[Eq 3.6, Prop 3.7]) showed that $r_{\phi,m}(x)$ is the value of the normalized Eisenstein series of weight $q^{m}-1$ on $\phi$ and established Deligne's congruence for Drinfeld modules, which ascertains for any $\p$ of degree $k \geq 1$ with $\Delta_\phi(x) \neq 0 \mod \p$ that 
\begin{equation}\label{deligne_congruence}
 h_{\phi,\p}(x) = r_{\phi,k}(x) \mod \p.
\end{equation}
Hence $r_{\phi,k}(x)$ is in a sense a lift to $A$ of all the Hasse invariants of $\phi$ at primes of degree $k$.\\ \\
In particular, if $\phi$ has supersingular reduction at a $\p$ of degree $k$, then $h_{\phi,\p}(x)=0$. By Deligne's congruence, $r_{\phi,k}(x) = 0 \mod \p$. From the recurrence \ref{eisenstein_recurrence}, it follows that $r_{\phi,k+1}(x) = 0 \mod \p$ since $\p$ divides $x^{q^k}-x$. Plugging $r_{\phi,k}(x) = r_{\phi,k+1}(x)= 0 \mod \p$ into the recurrence \ref{eisenstein_recurrence} yields  
\begin{equation}\label{supersingular_zero}
 r_{\phi,j}(x) = 0 \mod \p , \forall j \geq k.
\end{equation}
Likewise, if $\phi$ does not have supersingular reduction at a $\p$ of degree $k$, then by \cite{cor}[Lem 2.3]
\begin{equation}\label{supersingular_nonzero}
 r_{\phi,j}(x) \neq 0 \mod \p , \forall j \geq k.
\end{equation}
This suggests that we could use a Drineld module $\phi$ in a polynomial factorization algorithm to separate supersingular primes from those that are not. For most Drinfeld modules, the density of supersingular primes is too small for this to work. However, for a special class, Drinfeld modules with complex multiplication, the density of supersingular primes is $1/2$. 
\section{Drinfeld Modules with Complex Multiplication}\label{randomized_section}
A Drinfeld module $\phi$ is said to have complex multiplication by an imaginary quadratic extension $L/\F_q(x)$ if $End_{\F_q(x)}(\phi)\otimes_A \F_q(x) \cong L$. By imaginary, we mean the prime $(1/x)$ at infinity in $\F_q(x)$ does not split in $L$. For a $\phi$ with complex multiplication by $L/\F_q(x)$, a prime $\p$ that is unramified in $L/\F_q(x)$  is supersingular if and only if $\p$ is inert in $L/\F_q(x)$.\\ \\
This suggests the following  strategy to factor a monic square free polynomial $f(x) \in A$. Say $f(x)$ factors into monic irreducibles as $f(x) = \prod_i p_i(x)$. Pick a Drinfeld module $\phi$ with complex multiplication by some imaginary quadratic extension $L/\F_q(x)$. Compute $r_{\phi,k}(x) \mod (f(x))$ for some $k \leq \deg(f)$. By equations \ref{supersingular_zero} and \ref{supersingular_nonzero}, $$\gcd(r_{\phi,k}(x) \mod (f(x)), f(x)) = \prod_{(p_i)\ inert\ in\ L/\F_q(x), \deg(p_i)\leq k}p_i(x)$$ is a factor of $f$. Since for every degree, roughly half the primes of that degree are inert in $L/\F_q(x)$, the factorization  thus obtained is likely to be non trivial. 
\section{Randomized Polynomial Factorization using Drinfeld Modules with Complex Multiplication}
\subsection{Constructing Drinfeld Modules with Complex Multiplication}\label{drinfeld_construction_subsection}
Our strategy is to pick an $a \in \F_q$ at random and construct a Drinfeld module $\phi$ with complex multiplication by the imaginary quadratic extension $\F_q(x)(\sqrt{d(x)})$ of discriminant $d(x):=x-a$. From \cite{dor}, the Drinfeld module $\phi^\prime$ with $$g_{\phi^\prime}(x):=\sqrt{d(x)}+\left(\sqrt{d(x)}\right)^q, \D_{\phi^\prime}(x) := 1$$ has complex multiplication by $\F_q(x)(\sqrt{d(x)})$.\\ \\
However, $\phi^\prime$ has the disadvantage of not being defined over $A$ since $g_{\phi^\prime}(x) \notin A$.\\ \\
We construct an alternate $\phi$, that is isomorphic to $\phi^\prime$ but defined over $A$. The $J$-invariant \cite{gek} of $\phi^\prime$ is $$J_{\phi^\prime}(x) := \frac{g_{\phi^\prime}(x)^{q+1}}{\D_{\phi^\prime}(x)} = d(x)^{\frac{q+1}{2}}\left(1+d(x)^{\frac{q-1}{2}}\right)^{q+1}.$$
With the knowledge that two Drinfeld modules with the same $J$-invariant are isomorphic, we construct the Drinfeld module $\phi$ satisfying $$g_\phi(x)^{q+1} = (J_{\phi^\prime}(x))^2, \D_{\phi}(x)= J_{\phi^\prime}(x)$$ thereby ensuring $$J_{\phi}(x)=J_{\phi^\prime}(x).$$ Further, this assures that $\phi$ is defined over $A$ since

$$g_\phi(x):=d(x)(1+d(x)^{\frac{q-1}{2}}),\D_\phi(x):=d(x)^{\frac{q+1}{2}}(1+d(x)^{\frac{q-1}{2}})^{q+1}.$$
In summary, $\phi$ has complex multiplication by $\F_q(x)(\sqrt{d(x)})$) and is defined over $A$.
\subsection{Polynomial Factorization using Drinfeld Modules with Complex Multiplication}
We now state an iterative randomized algorithm to factor polynomials over finite fields using Drinfeld modules with complex multiplication. Curiously, it can be stated and implemented with no reference to Drinfeld modules.
\begin{algorithm}[H]
\caption{Polynomial Factorization}\label{factoring_algorithm}
\begin{algorithmic}[1]
\INPUT Monic square free $f(x) \in A$, positive integer $m \leq \deg(f).$
\OUTPUT Monic irreducible factors of $f(x)$ of degree at most $m$.
\State Perform root finding and output (and remove) all linear factors of $f(x)$.
\State Pick $a \in \F_q$ uniformly at random and set 
\LineComment $d(x):=x-a$.
\LineComment $g_\phi(x):=d(x)(1+d(x)^{\frac{q-1}{2}})$.
\LineComment $\D_\phi(x):=d(x)^{\frac{q+1}{2}}(1+d(x)^{\frac{q-1}{2}})^{q+1}$.
\State Initialize:
\LineComment $r_0(x):=  1 \mod f(x), r_1(x):=  g_\phi(x) \mod f(x)$. 
\LineComment Using multipoint evaluation, for $k \in \{1,2,\ldots,m\}$  compute
$$g_\phi(x)^{q^k} \mod f(x), \D_\phi(x)^{q^k} \mod f(x), x^{q^k} \mod f(x).$$
\LineComment Set $f_{ss}(x): = 1,f_{or}(x): =1 ,m_{ss}:=1 $ and $m_{or}:=1$.
\State For $k = 2$ to $m$
\LineComment Compute $$r_k(x) := g_\phi(x)^{q^{k-1}}r_{k-1}(x) - (x^{q^{k-1}}-x) \D_\phi(x)^{q^{k-2}} r_{k-2}(x) \mod f(x).$$
\LineComment If $\gcd(r_{k}(x),f(x))$ has degree $k$, mark it as an output.
\LineComment Else, $f_{ss}(x):= f_{ss}(x)\gcd(r_{k}(x),f(x))$ and $m_{ss} := k$.
\LineComment If $\gcd(x^{q^k}-x,f(x))/\gcd(r_{k}(x),f(x))$ has degree $k$, mark it as an output.
\LineComment Else, $f_{or}(x):= f_{or}(x)\gcd(x^{q^k}-x,f(x))/\gcd(r_{k}(x),f(x))$ and $m_{or} := k$.
\State Recursively call $(f_{ss}(x),m_{ss})$ and $(f_{or}(x),m_{or})$ to factor $f_{ss}(x)$ and $f_{or}(x)$.
\end{algorithmic}
\end{algorithm}
The assumption $\sqrt{q} \geq 100 n$ in Algorithm \ref{factoring_algorithm} can be made without loss of generality. For if $\sqrt{q} < 100 n$, we might choose to factor over a slightly larger field $\F_{q^\prime}$ where $q^\prime$ is the smallest power of $q$ such that $\sqrt{q^\prime} > 100 n$ and still recover the factorization over $\F_q$ (c.f. \cite[Remark 3.2]{nar}). Further, the running times are only affected by logarithmic factors.\\ \\
In Step $1$, all the linear factors of $f(x)$ are found and removed using a root finding algorithm. \\ \\
In Step $2$, we choose $a \in \F_q$ at random and construct a Drinfeld module $\phi$ with complex multiplication by $\F_q(x)(\sqrt{x-a})$. The primes that divide $\D_\phi(x)$ are precisely $\{(x-b),  b \in \F_q, \sqrt{b-a} \notin \F_q\} \cup \{(d(x))\}$. We might run into issues of bad reduction if the polynomial $f(x)$ to be factored had roots. It is to prevent this, we performed root finding in Step $1$.\\ \\
In Step $4$, from the recurrence \ref{eisenstein_recurrence}, it follows that the $r_k(x)$ computed coincides with the degree $k$ Hasse invariant lift $r_{\phi,k}(x)$. Hence at iteration $k$ in Step $4$, by Deligne's congruence \ref{deligne_congruence}, a degree $k$ monic irreducible factor $p(x)$ of $f(x)$ divides $r_k(x)$ if and only if $(p(x))$ is supersingular with respect to $\phi$. In particular, $\gcd(r_k(x),f(x))$ is the product of all degree $k$ irreducible factors of $f(x)$ that are supersingular with respect to $\phi$. If there is only one such factor, we output it. Else, the product is multiplied to $f_{ss}(x)$ to be split recursively later. Likewise, at iteration $k$ in Step $4$, $\gcd(x^{q^k}-x,f(x))/\gcd(r_{k}(x),f(x))$ is the product of all degree $k$ irreducible factors of $f(x)$ that are ordinary with respect to $\phi$. If there is only one such factor, we output it. Else, the product is multiplied to $f_{or}(x)$ to be split recursively later.\\ \\
The following Lemma \ref{splitting_lemma} states that any two distinct factors of $f(x)$ of the same degree are neither both supersingular nor both ordinary with probability $1/2$. This ensures that the splitting into supersingular and ordinary factors in Step $4$ is random enough that the recursion depth of our algorithm is logarithmic in $m$. The run time of our algorithm is dominated by the iteration in Step $4$ and the multipoint evaluation (c.f \cite{gs}) in Step $3$, both taking $\widetilde{O}(n^2 \log q + n \log^2 q)$ time.
\begin{lemma}\label{splitting_lemma}
Let $p_1(x),p_2(x) \in A $ be two distinct monic irreducible polynomials of degree $k$ where $1<k \leq \sqrt{q}$. Let $\phi$ be a Drinfeld module with complex multiplication by the imaginary quadratic extension $\F_q(x)(\sqrt{x-a})$ where $a \in \F_q$ is chosen at random. With probability close to $1/2$, exactly one of $(p_1(x))$ or $(p_2(x))$ is supersingular with respect to $\phi$.
\end{lemma}
\begin{proof}
Since $k > 1$ neither $(p_1(x))$ nor $(p_2(x))$ ramify in $\F_q(x)(\sqrt{x-a})$. Hence, the probability that exactly one of $(p_1(x))$ or $(p_2(x))$ is supersingular with respect to $\phi$ is precisely the probability that exactly one of $(p_1),(p_2)$ splits in $\F_q(x)(\sqrt{x-a})/\F_q(x)$.\\ \\
For $i\in\{0,1\}$, let $K_i:=\F_q(x)(\alpha_i)$ be the hyperelliptic extension of $\F_q(t)$ obtained by adjoining a root $\alpha_i$ of $y^2-p_i(x)$. By quadratic reciprocity over function fields \cite{car}, since $p_1(x)$ and $p_2(x)$ have the same degree, exactly one of $(p_1(x)),(p_2(x))$ splits in $\F_q(x)(\sqrt{x-a})$ if and only if $x-a$ is split in exactly one of $K_1,K_2$. That is, $(x-a)$ is neither completely split nor completely inert in the composite $K_1K_2$. Since $p_1(x)$ and $p_2(x)$ are distinct, $K_1$ and $K_2$ are linearly disjoint over $\F_q(x)$. Further, $K_1K_2$ is Galois over $\F_q(t)$ with $$Gal(K_1K_2/\F_q(x)) \cong Gal(K_1/\F_q(x)) \times Gal(K_2/\F_q(x)) \cong \Z/2\Z \oplus \Z/2\Z.$$  For $(x-a)$ to be neither totally split nor totally inert, the Artin symbol $$((x-a),K_1K_2/\F_q(x)) \in Gal(K_1K_2/\F_q(x))$$ has to be either $(0,1)$ or $(1,0)$ under the isomorphism $Gal(K_1K_2/\F_q(x)) \cong \Z/2\Z \oplus \Z/2\Z$. Applying Chebotarev's density theorem, the number $N$ of degree one primes $\{(x-a),a\in \F_q\}$ that are neither totally inert nor totally split in $K_1K_2$ is bounded by 
$$ \left|N - \frac{q}{2} \right| \leq 2 g(K_1K_2)\sqrt{q}$$
where $g(K_1K_2)$ is the genus of $K_1K_2$. By the Riemann-Hurwitz genus formula, $g(K_1K_2) = k-1 \leq \sqrt{q}/2$. Hence when $a \in \F_q$ is chosen at random, $(x-a)$ is neither totally inert nor totally split in $K_1K_2$ with probability close to $1/2$. 
\end{proof}
In summary, we have thus proven
\begin{theorem}
Algorithm \ref{factoring_algorithm} factors degree $n$ polynomials over $\F_q$ in expected time $\widetilde{O}(n^2 \log q + n \log^2 q)$.
\end{theorem}

\section{Deterministic Equal Degree Factorization using Drinfeld Modules}\label{deterministic_section}
In this section, we devise a deterministic algorithm for the equal degree factorization problem using Drinfeld modules with complex multiplication. Throughout this section, we assume that $q$ is a prime, the case most interesting to devise deterministic algorithms for.\\ \\
As in Algorithm \ref{factoring_algorithm}, we run in to difficulties of bad reduction if the polynomial to be factored has roots. Hence, as a preprocessing step, we extract and remove all the roots using a deterministic root finding algorithm \cite{sho}. Our run time dependence on $q$ is $\widetilde{O}(\sqrt{q})$, matching the best known existing algorithm of Shoup \cite{sho}. Our main result (Theorem \ref{deterministic_theorem}) in this section may thus be viewed as an alternative to Theorem 1 of Shoup \cite{sho}, particularly when the polynomial to be factored has at most one root.\\ \\
The proof techniques in this section were inspired by the proof of Lemma 3.2 in Shoup's deterministic algorithm \cite{sho}. The likeness of our proof of Lemma \ref{deterministic_lemma} to Lemma 5.2 in \cite{sho1} is noteworthy. 
\subsection{The Deterministic Equal Degree Factorization Algorithm}
We now derandomize Algorithm \ref{factoring_algorithm} in a natural manner by fixing an order in which the Drinfeld modules with complex multiplication are chosen. 
\begin{algorithm}[H]
\caption{Deterministic Equal Degree Factorization.}\label{deterministic_algorithm}
\textbf{Input:} The input to our algorithm is a monic square free polynomial $f(x) \in A $ of degree $n$ all of whose monic irreducible factors are of degree $k>1$.\\ \\ 
\textbf{Output:} Monic irreducible factors of $f(x)$.\\ \\
If $f(x)$ is irreducible, output $f(x)$. Else, iterate over $a \in \F_q$ in order $\{0,1,2,\ldots,q-1\}$. For each choice $a$, construct a Drinfeld module $\phi$ with complex multiplication by $\F_q(x)(\sqrt{x-a})$ as in \S~\ref{drinfeld_construction_subsection}. Compute the degree $k$ Hasse Invariant lift $r_{\phi,k}(x)$ using the recurrence \ref{eisenstein_recurrence} as in Algorithm \ref{factoring_algorithm}. If $\gcd(r_{\phi,k},f(x))$ is non trivial, split $f(x)$ and recursively apply our algorithm to the resulting factors. 
\end{algorithm}
Each iteration in Algorithm \ref{deterministic_algorithm} takes $\widetilde{O}(n^2 \log q + n \log^2 q)$ time from an analysis identical to Algorithm \ref{factoring_algorithm}. To bound the total running time, it suffices to bound the number of choices of $a$ attempted before a splitting occurs. For a splitting to never occur for $a \in \{0,1,\ldots,b\}$, there must exist two distinct monic irreducible factors $p_1(x)$ and $p_2(x)$ of $f(x)$ such that for each $a \in \{0,1,\ldots,b\}$, both $(p_1(x))$ and $(p_2(x))$ are both always either split or inert in $F_q(x)(\sqrt{x-a})/\F_q(x)$. 
For $i\in \{0,1\}$, $(p_i(x))$ splits in $\F_q(x)(\sqrt{x-a})/\F_q(x)$ if and only if $p_i(a)$ is a square. Hence, if $\chi$ denotes the quadratic character on $\F_q$, for a splitting to never occur for $a \in \{0,1,\ldots,b\}$, there must exist two distinct monic irreducible factors $p_1(x)$ and $p_2(x)$ of $f(x)$ such that $$\chi(p_1(a))\chi(p_2(a)) =1, \forall a \in \{0,1,\ldots,b\}.$$
The following Lemma \ref{deterministic_lemma} proves that $b$ is at worst $k\sqrt{q}\log q$. Knowing Lemma \ref{deterministic_lemma}, we can claim the following theorem since the total number of splittings required is at most $n/k$.
\begin{theorem}\label{deterministic_theorem}
Algorithm \ref{deterministic_algorithm} performs equal degree factorization of degree $n$ polynomials over a prime order field $\F_q$ in determinsitic $\widetilde{O}(n^3 \sqrt{q})$ time. 
\end{theorem}

\begin{lemma}\label{deterministic_lemma}
Let $q$ be an odd prime and $\chi$ the quadratic character on $\F_q$. Let $p_1(x),p_2(x) \in \F_q[x]$ be two monic irreducible polynomials of the same degree $d>1$. If $$\chi(p_1(a))\chi(p_2(a)) =1, \forall a \in \{z,z+1,\ldots,z+b-1\},$$
for some $z \in \F_q$, then $b \leq 2d \sqrt{q} \log q.$
\end{lemma}
\begin{proof}
Let $p_1(x),p_2(x) \in \F_q[x]$ be two monic irreducible polynomials of degree $d>1$ with 
\begin{equation}\label{quadratic_char_assumption}
\chi(p_1(a))\chi(p_2(a)) =1, \forall a \in \{z,z+1,\ldots,z+b-1\}
\end{equation}
for some $z \in \F_p$ and some positive integer $b$.\\ \\
Let $m$ be a positive integer less than $b$. Let $\mathcal{X}$ denote the affine $\F_q$ variety in the $m+1$ variables $Z,Y_0,Y_1,\ldots, Y_{m-1}$ defined by the system
\begin{equation}\label{system_equation}
\begin{cases}
Y_0^2=p_1(Z)p_2(Z) \\
Y_1^2=p_1(Z+1)p_2(Z+1) \\
Y_2^2=p_1(Z+2)p_2(Z+2)  \\
\ \ \ \ \ \vdots & \\
Y_{m-1}^2=p_1(Z+m-1)p_2(Z+m-1). 
\end{cases}
\end{equation}
Let $\mathcal{X}_m(\F_q)$ be the set of $\F_q$ rational points of $\mathcal{X}$ and $N_m:=|\mathcal{X}_m(\F_q)|$.\\ \\ 
By equation \ref{quadratic_char_assumption}, there is a point in $\mathcal{X}_m(\F_q)$ with $Z=z$. In fact there are $2^m$ points in $\mathcal{X}_m(\F_q)$ with $Z=z$. To see this, if $(z,y_0,y_1,\ldots,y_{m-1})$ is in $\mathcal{X}_m(\F_q)$, then so is $(z,-y_0,-y_1,\ldots,-y_{m-1})$. Further, for all $i\in\{0,1,\ldots,m-1\}$, $y_i \neq - y_i$, for otherwise $p_1(x)$ or $p_2(x)$ would have roots, contradicting their irreducibility.\\ \\
More generally, by equation \ref{quadratic_char_assumption}, for every $i \in \{0,1,2,\ldots,b-m-1\}$, there are $2^m$ points in $\mathcal{X}_m(\F_q)$ with $Z=z + i $. Thereby, we have the bound 
\begin{equation}\label{equation_bound_Nm}
(b-m)2^m \leq N_m.
\end{equation}
Our eventual objective is to bound $b$. To this end, we next bound $N_m$.\\ \\
For a $\bar{z} \in \F_q$, the number of points in $\mathcal{X}_m(\F_q)$ with $Z=\bar{z}$ is $$\prod_{i=0}^{m-1}\left(1 + \chi(p_1(\bar{z} + i)p_2(\bar{z} + i))\right).$$
Hence $$N_m = \sum_{\bar{z} \in \F_q} \prod_{i=0}^{m-1}\left(1 + \chi(p_1(\bar{z} + i)p_2(\bar{z} + i))\right)$$
$$ = \sum_{\bar{z} \in \F_q}  \sum_{e_0,e_1,\ldots,e_{m-1} \in \{0,1\}^m} \prod_{i=0}^{m-1}\left(\chi(p_1(\bar{z} + i)p_2(\bar{z} + i))\right)^{e_i}$$
\begin{equation}\label{N_equation}
=  \sum_{e_0,e_1,\ldots,e_{m-1} \in \{0,1\}^m}\sum_{\bar{z} \in \F_q} \chi\left( \prod_{i=0}^{m-1}(p_1(\bar{z} + i)p_2(\bar{z} + i))^{e_i}\right).
\end{equation}
For $e:=(e_0,e_1,\ldots,e_{m-1})\in \{0,1\}^{m}$, let $\ell_e$ denote $\sum_{i=0}^{m-1}e_i$ and define $$h_e(Z):= \prod_{i=0}^{m-1}(p_1(Z + i)p_2(Z + i))^{e_i} \in \F_q[Z].$$
For $e=(0,0,\ldots,0)$, $h_e(Z)=1$. For $e \neq (0,0,\ldots,0)$,
since $h_e(Z)$ is not a square and has degree at most $2d\ell_e$, by the Weil bound \cite{sch}, $$\sum_{\bar{z} \in \F_q} \chi(h_e(\bar{z})) \leq (2d\ell_e-1) \sqrt{q}$$
which when substituted in equation \ref{equation_bound_Nm} yields
$$ N_m =  \sum_{e \in \{0,1\}^m} \sum_{\bar{z} \in \F_q} \chi(h_e(\bar{z})) \leq q + \sqrt{q} \sum_{e \in \{0,1\}^m\setminus (0,0,\ldots,0)}(2d\ell_e-1).$$
$$= q+\sqrt{q}\sum_{\ell=1}^{m}\binom{m}{\ell}(2 d\ell -1) \leq q + \sqrt{q}(2^m-1)(2dm-1).$$
Substituting this bound for $N_m$ in equation \ref{N_equation}, we get $$(b-m)2^m \leq q + \sqrt{q}(2^m-1)(2dm-1).$$ We get to choose $m$ to optimize the bound on $b$. Setting $m= \lceil\log_2 \sqrt{q}\rceil$, since $d>1$ and $q >2$,
$$b \leq 2d \sqrt{q} \log q.$$

\end{proof}

\affiliationone{
   Anand Kumar Narayanan\\
   Computing and Mathematical Sciences\\
   California Institute of Technology\\
   \email{anandkn@caltech.edu}}
\end{document}